\documentclass [10pt,reqno]{amsart}
\usepackage{ucs}
\usepackage{amsmath, amssymb, amscd}
\usepackage{pb-diagram}
\usepackage[latin1]{inputenc}
\usepackage{amsmath}
\usepackage{amssymb}
\usepackage{epsfig}
\usepackage{graphicx}
\usepackage{psfrag}

\newtheorem{theorem}{Theorem}[section]
\newtheorem{definition}[theorem]{Definition}
\newtheorem{lemma}[theorem]{Lemma}

\newtheorem{prop}[theorem]{Proposition}
\newtheorem{corollary}[theorem]{Corollary}

\newtheorem*{remark}{Remark}

\renewcommand{\epsilon}{\varepsilon}

\DeclareMathAlphabet{\mathpzc}{OT1}{pzc}{m}{it}
\usepackage{mathrsfs}

\newcommand{\Z}{\mathbb{Z}}
\newcommand{\C}{\mathbb{C}}

\renewcommand{\qed}{$\hfill \square$ \smallskip \\}
\renewcommand{\phi}{\varphi}

\newcommand{\F}{\mathscr{F}}
\newcommand{\s}{\mathfrak{s}}

\renewcommand{\bar}{\overline}

\begin{document}
\thispagestyle{empty}
\title[Computable bounds for Rasmussen's concordance invariant]{Computable bounds for Rasmussen's concordance invariant}
\author{Andrew Lobb}

\begin {abstract} Given a diagram $D$ of a knot $K$, we give easily computable bounds for Rasmussen's concordance invariant $s(K)$.  The bounds are not independent of the diagram $D$ chosen, but we show that for diagrams satisfying a given condition the bounds are tight.  As a corollary we improve on previously known Bennequin-type bounds on the slice genus.
\end {abstract}

\address{Mathematics Department \\ Imperial College London \\ London SW11 7AZ \\ UK}
\email{a.lobb@imperial.ac.uk}

\maketitle

\section{Statement of results}

\subsection{Introduction}

In \cite{R}, Rasmussen defined a homomorphism on the smooth concordance group of knots $\mathcal{C}$

\[ s : \mathcal{C} \rightarrow 2\Z \rm{,} \]

\noindent which he showed had the property that

\[ |s(K)| \leq 2g^*(K) \]

\noindent where we write $g^* (K)$ for the smooth $4$-ball genus (or \emph{slice genus}) of $K$.

The starting point for this paper is the following Theorem of Rasmussen's \cite{R}:

\begin{theorem}
\label{raspos}
For positive knots $K$ (that is, knots which admit a diagram with no negative crossings)

\[ s(K) = 2g^*(K) \rm{.} \]
\end{theorem}

\noindent The point being that in the case of positive knots $K$, the computation of $s(K)$ is a triviality and agrees with twice the genus of an obvious candidate for a minimal-genus slicing surface (namely the one obtained by pushing the Seifert surface given by Seifert's algorithm into the $4$-ball).

The invariant $s(K)$ is equivalent to all the information contained in $\F ^j H^i (K)$, where $\F^j H^i$ is the perturbed version of standard Khovanov homology first defined and studied by Lee \cite{L}.  There is a spectral sequence with $E_2$ page being the standard Khovanov homology of a knot $K$ and $E_\infty$ page being the bigraded group $\F^jH^i(K)/ \F^{j+1}H^i(K)$ and many efforts to compute $s$ for knots other than for positive knots have made use of the existence of spectral sequences (for some nice examples see \cite{Sh}).

However, since it is known that $\F^j H^i (K) = 0$ for $i \not= 0$, to define $s(K)$ only requires knowledge of the partial chain complex

\[ \F^j C^{-1}(D) \stackrel{\partial_{-1}}{\rightarrow} \F^j C^{0}(D) \stackrel{\partial_{0}}{\rightarrow} \F^j C^{1}(D) \rm{,}\]

\noindent where $D$ is a diagram of $K$.  In fact, since explicit representatives for a basis of $\F ^j H^i (K)$ are known at the chain level, one only needs to know the map

\[ \partial_{-1} : \F^j C^{-1}(D) \rightarrow \F^j C^{0}(D) \rm{.} \]

\begin{remark}
For a positive diagram $D$, $C^{-1}(D) = 0$.  This is what made Theorem \ref{raspos} a trivial corollary once the properties of $s$ were established.
\end{remark}

By studying this map we obtain a diagram-dependent upperbound $U(D)$ for $s(K)$.  We also give an error estimate $2\Delta(D)$ for this upperbound.  The resulting lowerbound $U(D) - 2\Delta(D)$ for $s(K)$ improves upon previously known Rudolph-Bennequin-type inequalities.  We give a list of particular cases where $\Delta(D)$ vanishes and so $U(D)$ necessarily agrees with $s(K)$.

\subsection{Results}

The following results are stated for knots, since the Rasmussen invariant is most familiar in this setting.  Some results however admit a generalization to links (via the definition of $s$ for links as found for example in \cite{BW}).  We discuss this in Section \ref{linksection}.

Our results concern an easily-computable number $U(D) \in 2\Z$ which is defined from an oriented knot diagram $D$.  Postponing an explicit description of how to compute $U(D)$ until Definition \ref{maindefn}, we begin by giving some results.

\begin{theorem}
\label{upperbound}
\[ s(D) \leq U(D) \rm{.}\]
\end{theorem}

Of course, we must remember that $s(D)$ depends only on the isotopy class of the knot represented by $D$, whereas the same is not true of $U(D)$.  Hence in order for the bound of Theorem \ref{upperbound} to be a good bound, we should expect to be forced to give some restrictions on diagrams $D$:

\begin{prop}
\label{posandneg}
The bound of Theorem \ref{upperbound} is tight for positive diagrams $D$ and for negative diagrams $D$.
\end{prop}

\begin{prop}
\label{braid}
Let $\epsilon_i \in \{ -1, +1 \}$ for $i = 1, 2, \ldots, n$.  Then if $w$ is any word in the $n$ letters

\[ \{ \sigma_1^{\epsilon(1)}, \sigma_2^{\epsilon(2)}, \ldots, \sigma_n^{\epsilon(n)} \} \]

\noindent and $B$ is a knot diagram which is the closure of the $(n+1)$-stranded braid represented by $w$, then we have

\[ s(B) = U(B) \rm{.} \]
\end{prop}

\begin{remark}
We note that knots admitting such a braid presentation are known to be fibered \cite{stallings}, so in particular not every knot admits such a presentation.
\end{remark}

\begin{prop}
\label{alternating}
Let $D$ be an alternating diagram of a knot.  Then we have

\[ s(D) = U(D)\rm{.} \]
\end{prop}

Propositions \ref{posandneg}, \ref{braid}, and \ref{alternating} are each consequences of Theorem \ref{trees} for which we need a few definitions.  Given a diagram $D$ we write $O(D)$ for the oriented resolution.

\begin{definition}
We form a decorated graph $T(D)$, known as the Seifert graph of $D$, as follows:

We start with a node for each component of $O(D)$.  Each crossing in $D$, when smoothed, lies on two distinct components of $O(D)$; for each positive (respectively negative) crossing of $D$ we connect the corresponding nodes by an edge decorated with $+$ (respectively $-$).
\end{definition}

Note that $T(D)$ by itself is not enough to recover the full Khovanov chain complex of the diagram $D$, but if we added extra data of an ordering of the edges at each node, we would be able to recover the full complex.

\begin{definition}
From $T(D)$ we now form two other graphs:

We form a subgraph $T^- (D)$ (respectively $T^+(D)$) from $T(D)$ by removing all edges of $T(D)$ decorated with a $+$ (respectively $-$).
\end{definition}

\begin{definition}
\label{maindefn}
We define the number
\[ U(D) = \#\rm{nodes}(T(D)) - 2\#\rm{components}(T^-(D)) + w(D) + 1 \rm{,}\]

\noindent where $w(D)$ is the writhe of $D$.
\end{definition}

\begin{definition}
\label{errordefn}
We define the number
\[ \Delta(D) = \#\rm{nodes}(T(D)) - \#\rm{components}(T^-(D)) - \#\rm{components}(T^+(D)) + 1 \rm{.}\]
\end{definition}

Then we have

\begin{theorem}
\label{trees}
If $\Delta(D) = 0$ then $s(D) = U(D)$.  In fact we can say more:

\[ U(D) - 2\Delta(D) \leq s(D) \leq U(D) \rm{.} \]
\end{theorem}

Theorem \ref{trees} enables us to improve on previously known easily-computable combinatorial lower bounds for the slice genus.  We have:

\begin{corollary}
\label{benny}
\begin{eqnarray*} 2g^*(K) &\geq& s(K) \geq U(D) - 2\Delta(D) \\
&\geq& w(D) - \#\rm{nodes}(T(D)) + 2\#\rm{components}(T^+(D)) -1 \rm{,}
\end{eqnarray*}
\end{corollary}

\noindent which is stronger than the Rudolph-Bennequin inequalities as proved in \cite{K2}, \cite{P}, and \cite{Sh} (for a nice discussion see \cite{St}).

\begin{proof}{of Propositions \ref{posandneg}, \ref{braid}, and \ref{alternating}.}
This is just a matter of checking that the condition $\Delta(D) = 0$ of Theorem \ref{trees} holds in each case.  This is only a non-trivial check for the case of $D$ being alternating.

Suppose $D$ is an alternating diagram.  The complement of the oriented resolution $O(D)$ is a number of regions of the plane.  If $D$ is not the trivial diagram, there is a unique way to associate to each region either a $+$ or a $-$ such that only positive (respectively negative) crossings of $D$ occur in regions associated with a $+$ (respectively $-$) and such that adjacent regions have different associated signs.  See Figure \ref{alt} for an example.

Then each region with associated sign $+$ (respectively $-$) corresponds to exactly one component of $T^+(D)$ (respectively $T^-(D)$).  Since there is one more region than there are circles of $O(D)$ (or equivalently nodes of $T(D)$) we must have $\Delta(D) = 0$.
\end{proof}

\begin{figure}
\centerline{
{
\psfrag{pnodes}{$2p - 1$ nodes}
\psfrag{+}{$+$}
\psfrag{-}{$-$}
\psfrag{ldots}{$\ldots$}
\psfrag{T(D)}{$T(D)$}
\psfrag{T-(D)}{$T^-(D)$}
\psfrag{T+(D)}{$T^+(D)$}
\includegraphics[height=5in,width=3.7in]{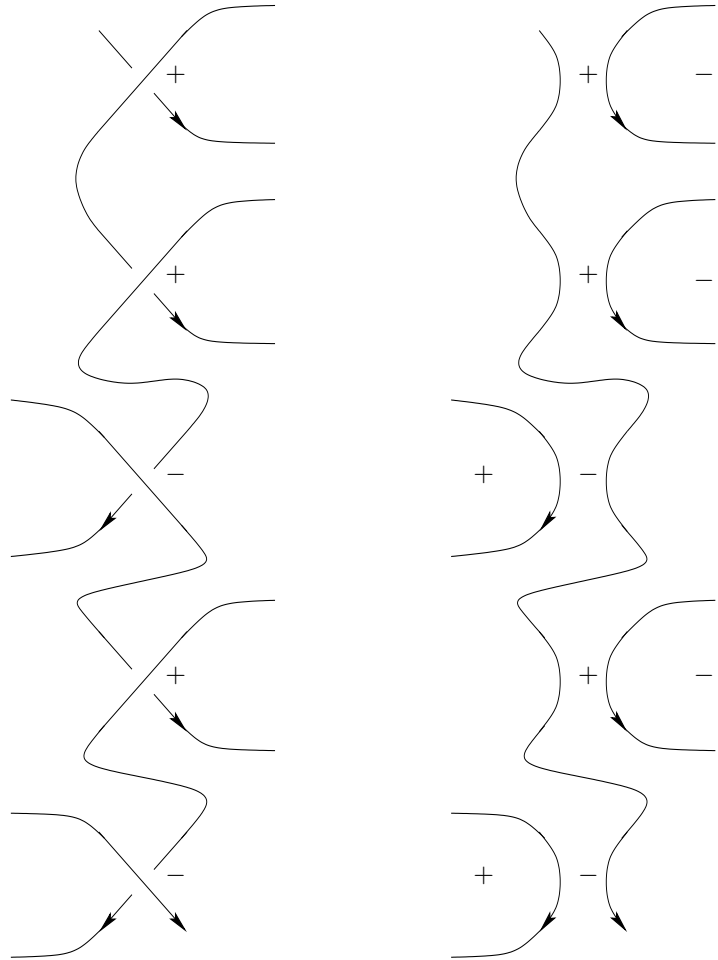}
}}
\caption{On the left of this figure we show part of an alternating knot diagram $D$.  We indicate which crossings are positive and which negative.  On the right of the figure is the oriented resolution $O(D)$ on which we indicate how to uniquely associate $+$ or $-$ to each component of the complement of $O(D)$.}
\label{alt}
\end{figure}

We note that Proposition \ref{alternating} gives a combinatorial formula for the Rasmussen invariant of an alternating diagram.  It is known \cite{L} that the Rasmussen invariant of an alternating knot agrees with the signature of the knot, and there is also known \cite{Tr} a combinatorial formula for the signature of an alternating diagram.  Proposition \ref{alternating} gives an equivalence between these two results.

%\begin{remark}
%In fact for braids, the conditions of Proposition \ref{braid} are \emph{equivalent} to the vanishing of $\Delta(D)$.
%\end{remark}

There is a nice topological interpretation of $\Delta$ which is useful in computing it by hand:

\begin{prop}
Form a graph $G$ which has a node for each component of $T^-(D)$ and a node for each component of $T^+(D)$.  Each circle in $O(D)$ is a member of exactly one component of $T^-(D)$ and exactly one component of $T^+(D)$; for each circle in $O(D)$ let $G$ have an edge connecting the corresponding pair of nodes.

Then $\Delta(D) = b_1(G)$, the first betti number of $G$.
\end{prop}

\begin{proof}
This follows from the connectedness of $G$ so that we have

\begin{eqnarray*}
b_1(G) &=& b_0(G) - \chi (G) = 1 - \#\rm{nodes}(G) + \#\rm{edges}(G) \\
&=& 1  - \#\rm{components}(T^-(D)) - \#\rm{components}(T^+(D)) +  \#\rm{nodes}(T(D)) \\
&=& \Delta(D) \rm{.}
\end{eqnarray*}
\end{proof}

Just prior to posting on the arXiv, we heard from Tomomi Kawamura \cite{K1} that she has independently obtained several of the results in this paper, using entirely different methods.  We thank Tetsuya Abe and Cornelia van Cott for their comments on an earlier draft of this paper.

\section{Proof of main results}

We assume familiarity with the definition of the Khovanov chain complex defined from a knot diagram $D$, and with Rasmussen's paper \cite{R}.  We write $\F^j C^i (D)$ for Lee's perturbed chain complex with complex coefficients (where the TQFT is induced from the Frobenius algebra $\C \hookrightarrow \C[x]/(x^2 - 1)$), with the $\F^j$ representing the quantum filtration:

\[ \ldots \subseteq \F^{j+1}C^i  \subseteq \F^{j}C^i \subseteq \F^{j-1}C^i \subseteq \ldots \rm{,} \]

\noindent and the superscript $i$ denoting the homological grading:

\[ \partial_i : \F^j C^i \rightarrow \F^j C^{i+1}, \partial_i \partial_{i-1} = 0 \rm{.} \]

\noindent Similarly we write $\F^jH^i(D)$ for the homology of the chain complex $\F^jC^i(D)$.

There is a distinguished subspace of $C^0(D)$ which I shall write as $H(O(D)) \{ w(D) \}$; $O(D)$ being the oriented resolution of $D$ and $\{ w(D) \}$ being a shift in the quantum filtration by the writhe of $D$.  Here one can think either of $H$ as being Lee's TQFT functor or of $H(O(D))$ as being the perturbed Khovanov homology of the ($0$-crossing) diagram $O(D)$.

By Lee \cite{L} we know that

\begin{theorem}
Given a knot diagram $D$ with orientation $o$, there exist $\s_o, \s_{\bar{o}} \in H(O(D))\{ w(D) \} \leq C^0(D)$ such that $\partial_0 \s_o = \partial_0 \s_{\bar{o}} = 0$  Furthermore, the homology $\F^jH^i(D)$ is 2-dimensional and supported in homological grading $i=0$ with $H^0 (D) = < [\s_o] , [\s_{\bar{o}}] >$.
\end{theorem}

There is an explicit description of these generators at the chain level:

\begin{definition}
\label{generators}
The orientation $o$ on $D$ induces an orientation on $O(D)$.  For each circle $C$ in $O(D)$ we give a invariant which is the mod 2 count of the number of circles in $O(D)$ separating $C$ from infinity, to which we add $0$ (respectively $1$) if $C$ has the counter-clockwise (respectively clockwise) orientation.  We label $C$ with $v_- + v_+$ (respectively $v_- - v_+$) if the invariant is $0$ (respectively $1$) $\pmod{2}$.  This determines an element $\s_o \in H(O(D))\{ w(D) \}$, $\s_{\bar{o}}$ being given in the same way using the opposite orientation $\bar{o}$ on $D$.
\end{definition}

We know that, in Rasmussen's notation, $s(D)=s_{\rm{min}}(D) + 1$ and $s_{\rm{min}}(D)$ is the filtration grading of the highest filtered part of $H^0(D)$ to contain $[\s_o]$ (or equivalently $[\s_{\bar{o}}]$ - this interchangeability is taken as understood from now on).  This is the same as the filtration grading of the highest filtered part of $C^0 / im (d_{-1})$ containing $[\s_o]$.  It follows that

\begin{lemma}
\label{ronen}
Let $p : C^0(D) \rightarrow H(O(D)) \{ w(D) \}$ be the projection onto the vector space summand.  Then

\[ s_{\rm{min}}(D) \leq L(D) \]

\noindent where $L(D)$ is the filtration grading in $H(O(D)) \{ w(D) \}/im(p \circ d_{-1})$ of the highest filtered part containing $[s_o]$.  \qed
\end{lemma}

%%We write $B$ for the standard basis of $C^0_{O(D)}$ obtained by choosing either $v_+$ or $v_-$ for each component of $O(D)$.
%%\begin{lemma}
%%For any element $e$ of $B$.
%%\end{lemma}

\begin{proof}(of Theorem \ref{upperbound})
Given a knot diagram $D$ with orientation $o$, we write $n_+$, $n_-$ for the number of positive, negative crossings of $D$ respectively so that the writhe $w(D) = n_+ - n_-$.  Form the diagram $D^-$ by taking the oriented resolution at each of the positive crossings.  Note that diagram $D^-$ is also oriented with writhe $-n_-$.  Suppose there are $l$ components $D_1^-, D_2^-, \ldots, D_l^-$ of $D^-$ (where we mean components as a subset of the plane, so that the standard 2-crossing diagram of the Hopf link would be considered as a single component, for example) and suppose that $D_r^-$ has $n_r$ crossings for $1 \leq r \leq l$.

We observe that, up to quantum filtration shift by $\{ n_+ \}$, the map

\[ p \circ d_{-1} : C^{-1}(D) \rightarrow H(O(D)) \{w(D)\} \leq C^0(D) \]

\noindent can be identified with the map

\[ d_{-1} : C^{-1}(D^-) \rightarrow C^{0}(D^-) = H ( O(D^-) ) \{ -n_- \} \rm{.} \]

\noindent This latter map is in fact $\bigoplus_{r=1}^l d_{-1}^r \otimes \rm{id}^r$ where

\[ d_{-1}^r : C^{-1} ( D_r^{-} ) \rightarrow C^{0} ( D_r^{-} ) = H ( O(D_r^-) ) \{ -n_r \} \rm{,} \]

\noindent is the $(-1)$th differential in the chain complex $C^*(D_r^{-})$ and

\[ id^r : H(O(D^- \setminus D^-_r)) \{-n_- + n_r \} \rightarrow H(O(D^- \setminus D^-_r)) \{-n_- + n_r \} \]

\noindent is the identity map.

Inductively on $r$ we observe a canonical identification

\begin{eqnarray*}
\rm{coker}(\bigoplus_{r=1}^l (d_{-1}^r \otimes \rm{id}^r)) &=& \bigotimes_{r=1}^l \rm{coker}(d_{-1}^r) \\
&=& \bigotimes_{r=1}^l (H^0(D_r^-)) \rm{.}
\end{eqnarray*}

Now $\s_o = \s_1 \otimes \s_2 \otimes \cdots \otimes \s_l$, where $\s_r \in C^0(D_r^-)$ is either the element $\s_{o'}$ or $\s_{\bar{o'}}$ where we use $o'$ to stand for the induced orientation on the oriented resolution of $D_r^-$.  This is because the mod 2 invariant associated to each circle $C \subset O(D_r^{-})$ via Definition \ref{generators} differs by $0$ or $1$ from the invariant associated to $C \subset O(D)$ via Definition \ref{generators}, and it is the same difference for all circles of $O(D_r^-)$.

%%Now, the image of $\s_o \in H(O(D)) \{w(D)\} \leq C^0(D)$ under the projection map $H(O(D)) \{w(D)\} \rightarrow C^0 (D_r^{-})$ is either the element $\s_{o'}$ or $\s_{\bar{o'}}$ where we use $o'$ to stand for the induced orientation on the oriented resolution of $D_r^-$.  This is because the mod 2 invariant associated to each circle $C \subset O(D_r^{-})$ via Definition \ref{generators} differs by $0$ or $1$ from the invariant associated to $C \subset O(D)$ via Definition \ref{generators}, and it is the same difference for all circles of $O(D_r^-)$.

Suppose the number of components of $O(D_r^-)$ is $e_r$.  We observe that $\F^{e_r - n_r} C^0 (D_r^{-})$ is the highest filtered part of $C^0 (D_r^-)$ to be non-zero and is $1$-dimensional.  Therefore by Lemma 3.5 \cite{R}, $[\s_r]$ has filtered degree less than or equal to $e_r - n_r -2$ in $H^0(D_r^-)$.

We compute for $L(D)$ in Lemma \ref{ronen}:

\begin{eqnarray*}
L(D) &=& n_+ + \sum_{r=1}^l e_r - n_r -2 \\
&=& n_+ -n_- + \#\rm{nodes}(T(D)) - 2\#\rm{components}(T^-(D)) \\
&=& \#\rm{nodes}(T(D)) - 2\#\rm{components}(T^-(D)) + w(D) \rm{.}
\end{eqnarray*}

Hence we have

\begin{eqnarray*}
s(D) &=& s_{\rm{min}}(D) + 1 \leq L(D) + 1 \\
&=&  \#\rm{nodes}(T(D)) - 2\#\rm{components}(T^-(D)) + w(D) + 1 = U(D) \rm{.}
\end{eqnarray*}
\end{proof}

\begin{proof}(of Theorem \ref{trees})
Given an oriented knot diagram $D$, let $\bar{D}$ be the mirror image of $D$.  It is then easy to check that

\[ 2\Delta(D) = U(D) + U(\bar{D}) \rm{.} \]

So we have

\begin{eqnarray*}
s(D) = -s(\bar{D}) \geq -U(\bar{D}) = U(D) - 2\Delta(D) \rm{.}
\end{eqnarray*}
\end{proof}

\section{Generalizations to links}
\label{linksection}

Given an $r$-component link $L \subset S^3$, let $G(L)$ be the genus of a connected minimal-genus smooth surface in the $4$-ball which has $L$ as boundary.  We extend the definition of the slice genus $g^*$ to links by defining

\[ g^*(L) = G(L) + \frac{1}{2} - \frac{r}{2} \in \frac{1}{2}\Z \rm{.} \]

The definition of the $s$-invariant for links as found in \cite{BW} is such that the proof of Theorem \ref{upperbound} carries through unchanged to this setting.  Also by \cite{BW} we know that

\begin{enumerate}
\item $s(L) \leq 2g^*(L) \rm{,}$
\item $s(L) + s(\bar{L}) \geq 2-2r \rm{.}$
\end{enumerate}

Hence we also obtain a version of Corollary \ref{benny} for links:

\begin{corollary}
Suppose $D$ is a diagram of an $r$-component link and $T(D)$ and $T^+(D)$ are the associated graphs, then

\[ 2g^*(D) \geq w(D) - \#\rm{nodes}(T(D)) + 2\#\rm{components}(T^+(D))  - 2r + 1 \rm{.} \]
\end{corollary}

\begin{proof}
We have

\begin{eqnarray*}
2g^*(D) &\geq& s(D) \\
&\geq& 2 - 2r - s(\bar{D}) \\
&\geq& 2 - 2r - U(\bar{D}) \\
&=& w(D) - \#\rm{nodes}(T(D)) + 2\#\rm{components}(T^+(D))  - 2r + 1  \rm{.}
\end{eqnarray*}
\end{proof}

\end{document}